\documentclass[12pt]{amsart}
\usepackage{latexsym,fancyhdr,amssymb,color,amsmath,amsthm,graphicx,listings,comment}
\usepackage[section]{placeins} 
  \newtheorem{propo}{Proposition} \newtheorem{coro}{Corollary}
\definecolor{red1}{rgb}{1,0.9,0.9} \definecolor{blue1}{rgb}{0.9,0.9,1} \definecolor{green1}{rgb}{0.9,1,0.9}
\def\conjecture#1{ \vspace{2mm} \begin{center} \fcolorbox{green1}{green1}{ \parbox{13.2cm}{{\bf Question:} #1}} \vspace{2mm} \end{center} }
\let\paragraph\subsection

\title{Integral geometric Hopf conjectures}
\author{Oliver Knill}
\date{January 5, 2020}
\address{Department of Mathematics \\ Harvard University \\ Cambridge, MA, 02138 }

\begin{document}
\maketitle

\begin{abstract}
The Hopf sign conjecture states that a compact Riemannian 2d-manifold $M$ of positive curvature has
Euler characteristic $\chi(M)>0$ and that in the case of negative curvature
$\chi(M) (-1)^d >0$. The Hopf product conjecture asks whether a positive curvature metric
can exist on product manifolds like $S^2 \times S^2$. By formulating curvature integral geometrically, 
these questions can be explored for finite simple graphs, 
where it leads to linear programming problems. In this more expository document we aim to explore also
a bit of the history of the Hopf conjecture and mention some strategies of attacks which have been tried. 
We illustrate the new integral theoretic $\mu$ curvature
concept by proving that for every positive curvature manifold $M$ there exists
a $\mu$-curvature $K$ satisfying Gauss-Bonnet-Chern $\chi(M)=\int_M K \; dV$ such that $K$ 
is positive on an open set $U$ of volume arbitrary close to the volume of $M$. 
\end{abstract}

\section{Introduction}

\paragraph{}
The {\bf Hopf sign conjecture} from the 1930ies is an important
open problem in the field of global Riemannian geometry, where local quantities like curvature are linked to global quantities
like Euler characteristic. It is motivated first of all from the case $d=1$ where it follows from Gauss-Bonnet.
The sign is explained that products $M=M_1 \times \cdots \times M_d$ of negatively curved 2-manifolds 
which leads to a non-positively curved $M$ with $\chi(M) (-1)^d >0$. A first line of attack came through
Gauss-Bonnet-Chern and Milnor's remark that the Gauss-Bonnet-Chern theorem works for 4-manifolds \cite{Chern1966}.
The full {\bf algebraic Hopf conjecture} trying to prove positive Gauss-Bonnet-Chern integrand from positive curvature
was explored more in \cite{Weinstein71} but counter 
examples due of Geroch in 1976 \cite{Geroch,Klembeck} indicated that the Gauss-Bonnet-Chern integrand $K$ 
can become negative even in the positive curvature case. 

\paragraph{}
An other line of attack came by exploring examples of positive or non-negative curvature manifolds aiming to 
enlarge the class of cases supporting the conjecture and also to search for cases, where it might fail. 
All simply connected large dimensional 2d-manifolds are known to be spheres and projective spaces suggesting
even a much more daring sphere theorem claiming that all large dimensional simply connected positive curvature
$2d$-manifolds are spheres or projective spaces. As for the product conjecture,
all the known obstructions to in the simply connected case already are present for non-negative curvature. This
suggests that positive curvature can be deformed from non-negative curvature.
Surveys are given in \cite{Ziller,Ziller2,EscherZiller} and \cite{Eberlein} for non-positive curvature.
An other approach to the conjectures is to assume that $M$ has more symmetry like homogeneous spaces. 
Other symmetry programs were started by Grove in the 90ies \cite{Grove2009}. It has led to many 
results like \cite{PuettmanSearle,Kennard2013}. 

\paragraph{}
An other line of attack is to find bounds on Betti numbers $b_k(M)$. An early result by Gromov in the non-positive curvature
case is that there is a global bound on the Betti numbers $b_k$ which only depends on dimension. 
To get positive $\chi(M)$, one only needed to establish zero
Betti numbers $b_k=0$ for odd $k$ a fact which by Morse inequalities $b_k \leq c_k$ could be done 
by estimating the number $c_k$ of critical points of a cleverly chosen Morse function with Morse index $m(x)=k$. 
An other approach to exclude cohomology is Bochner's method using Hodge theory, 
analyzing harmonic $k$-forms and excluding them by excluding maxima of curvatures of 
such form as mentioned in \cite{Chern1966}.
The conjecture was studied also on subclasses like manifolds with strongly positive 
curvature. The analogue question there is called the ``weak Hopf question" \cite{BettiolMendes}, which 
however turns also to be inaccessible by the algebraic Gauss-Bonnet integrand. 

\paragraph{}
One can hope to construct a Morse function $f$ with only positive
Poincar\'e-Hopf indices $i_f(x) = (-1)^{m(x)}$ and get positive Euler characteristic
also without Morse inequalities and Euler-Poincar\'e $\chi(G)=\sum_k (-1)^k b_k$. 
This rather simple approach amounts to exclude critical points with odd Morse index $m(x)$, 
if $M$ is a positive curvature $2d$ manifold. This reduces to a local problem:
if $x$ is a critical point of a Morse function with Morse index $m(x)=k$, then by the Morse lemma, 
the $(k-1)$ manifold $R_f(x)=S_r(x) \cap W_f^-(x)$ with $k$-dimensional
stable manifold $W^-(x)$ is the product of two spheres $S^{k-1} \times S^{2d-1-k}$. For $k=1$ for example, 
$R_f(x)$ is then disconnected which can be excluded in the orientable positive curvature case. 
Having $c_1=0$ and so $b_1=0$, one so gets a weak Synge result. For suitable Morse functions 
coming from embeddings into larger dimensional Euclidean spaces,
one can do better as as can be illustrated in the case of a hypersurface $M \subset E$, where linear 
functions in the ambient space $E$ produce good Morse functions on $M$. 

\paragraph{}
Rather than insisting on finding a specific Morse function $f$ without odd Morse indices
and so get positive Poincar\'e-Hopf indices $i_f(x)$,
one can try to average the indices over a probability space of Morse functions. 
The example of the projective plane for which every
Morse function also has a negative Poincar\'e-Hopf index ($\chi(M)=1$ and $i_{f}({\rm max})=i_f({\rm min})=1$
implies there is an other $x$ with $i_f(x)=-1$), shows that we need in general more than one
Morse function. So, maybe, some probability space $(\Omega,\mathcal{A},\mu)$
of Morse functions produces indices $i_f(x)$ which as expectation $K(x)=E[i_f(x)]$ produces
a {\bf curvature} $K$ which is positive everywhere? If $M$ is Nash embedded into an Euclidean space $E$
and the probability space of linear functions $f_v(x) = v \cdot x$ with $|v|=1$ is taken with the natural
volume measure of the sphere in $E$ and so on $\Omega$ then this leads to the Gauss-Bonnet-Chern integrand. 
As Geroch's example shows, this in general does not produce $K(x)>0$ but index expectation produces 
more possibilities. 

\paragraph{}
In some sense, Chern's 1944 approach to the Gauss-Bonnet-Chern theorem was of an integral geometric nature already
using Poicar\'e-Hopf. Chern's proof is sketched for example in \cite{KobayashiNomizu}:
the Euler curvature $K$ is lifted to a curvature $p^*K$ on the unit sphere bundle $p:N \to M$, where it is
$d\pi$ for a $(2d-1)$-form $\pi$
on $N$ such that $\int_S \pi =1$ for every fiber $S$. If $X$ is a vector field on $M$ with finitely many singularities $x_k$
of index $i(x_k)$, it defines a submanifold $U$ in $N$ with boundary $\delta U = i(x_k) S_k$. Now, using Stokes theorem one has
$\chi(M)=\sum_k i(x_k) = \sum_k i(x_k) \int_{S_k} \pi = \int_{\delta U} \pi = \int_U d\pi = \int_{S(M)} p^*(K) = \int_M K$. 
Such acrobatics is not necessary in the discrete, where the curvature $K$ can just be obtained by integrating over
all possible Poincar\'e-Hopf indices $i_f(x)$, where $f$ runs over all possible coloring of the graph. In some sense, 
the curvature is collected and concentrated around critical points $x_k$ of a vector field, where it produces a
divisor, an integer valued point measure. In the discrete this works much easier: the curvature has been obtained
by distributing the energy $\omega(x) = (-1)^{{\rm dim}(x)}$ of a simplex $x$ to vertices. By distributing 
$\omega(x)$ instead to an assigned vertex $v$ contained in $x$ (we call this a vector field in \cite{MorePoincareHopf}), 
we get a Poincar\'e-Hopf divisor. The idea is conceptually close to Chern but in the continuum, we have no 
individual simplex elements to play with and must refer to sheaf theoretical constructs like differential forms. 

\paragraph{}
When curvature is reformulated integral geometrically it 
leads to curvatures which can be used both for Riemannian manifolds $M$ as 
well as for finite simple graphs $G$. The sectional curvatures is then defined 
as index expectation of the Morse functions $f$ restricted to $2$-dimensional geodesic sub-spaces. 
The tricky thing  in the discrete is to prescribe exactly what we mean with ``geodesic wheel subspaces"
in the graph theory case. 
To define curvature we also need a measure $\mu$ on Morse functions $f$ on a manifold $M$ or the set of locally injective functions 
on a graph $G$. Each such measure produces not only curvature $E[i_f(x)]$ generalizing the 
Gauss-Bonnet-Chern integrand, it also defines sectional curvatures by looking at index expectation
on geodesic discs $\exp_x(D)$ of a Riemannian manifold $M$ or on ``geodesic wheel sub-graphs"' 
of $G$. This allows to express integral geometrically what ``positive curvature" means. 

\paragraph{}
The Hopf sign conjecture claiming that an even-dimensional 
positive curvature manifold has positive Euler characteristic or that the Cartesian product 
$S^2 \times S^2$ of two $2$-spheres $S^2$ carries a positive curvature metric can now be explored
in graph theory for even-dimensional $d$-graphs. The set-up also revives the algebraic Hopf 
conjecture statement. Does positive sectional curvatures for some measure $\mu$ 
on Morse functions or colorings imply that for some other measure $\nu$, the $\nu$-curvature $K_{\nu}$ 
is positive everywhere on $M$? A positive answer would imply the Hopf conjecture 
because of Gauss-Bonnet $\chi(G)=\int_M K_{\nu} dV$. 
For classical Riemannian manifolds, the approach produces slightly stronger conjectures 
with more flexible statements claiming that every positive $\mu$-curvature manifold has positive Euler 
characteristic and that there is no measure $\mu$ on Morse functions such that $S^2 \times S^2$ has positive 
sectional $\mu$-curvature. 

\paragraph{}
Before ending the introduction, we should point out that there is no lack of proposals for 
notions of curvature in the discrete and especially in graph theory. Beside the combinatorial
curvatures going back to the 19th century (probably first by Eberhard \cite{Eberhard1891} 
which lead to the Levitt curvature \cite{Levitt1992,cherngaussbonnet} which satisfies
the full Gauss-Bonnet relation (where graphs are considered
equipped with the Whitney simplicial complex structure)
There are various proposals for Ricci curvatures
like \cite{Forman2003}. Then there are notions which depend on probability \cite{Ollivier,JostLiu}. 
A recent new approach is \cite{SaucanSamalJost}. What appears to be new with the integral 
geometric proposal is that also the theorems should parallel the results in the continuum. 
For us, the Gauss-Bonnet-Chern result is crucial. 
We illustrate this here in the concept of the Hopf conjecture but one could look at 
from the point of view related to other topics like sphere theorems.

\section{In a nutshell}

\paragraph{}
The following integral theoretic notion of curvature works both for Riemannian manifolds as well as for 
finite simple graphs: if $\Omega$ is the set of Morse functions on a compact differentiable manifold $M$
or the set of locally injective functions on the vertex set of a finite simple
graph $G$, we can equip $\Omega$ with a $\sigma$-algebra and probability measure $\mu$ 
and define the {\bf sectional $\mu$-curvature}
$K_H(x)$ of a sub-graph $H$ of $G$ or a sub-manifold $H$ as the Poincar\'e-Hopf index expectation 
$K_H(x) = {\rm E}_H[i_f(x)]$ with respect to the induced measure $\mu_H$ 
on colorings respectively Morse functions on $H$. The Poincar\'e-Hopf formula 
$\chi(H) = \sum_{x} i_f(x)$ for $f \in \Omega$ leads then to the
Gauss-Bonnet property $\chi(H) = \int_H K(x) dV(x)$ or $\chi(H) = \sum_{x \in H} K(x)$ in the discrete.
We have looked already at the question to realize constant index expectation in 
\cite{ConstantExpectationCurvature}.

\paragraph{}
In the manifold case, the sectional curvature $K_H(x)$ is defined as the Gauss curvature 
of $H=\exp_x(D)$ at $x$ if $D \subset T_xM$ is a $2$-dimensional disk in the tangent space $T_xM$ 
and $\exp_x$ is the exponential map. In the case of $d$-graphs, a sectional curvature 
of $G$ at $x$ is defined now as the curvature of the center $x$ in an embedded 
wheel graph $H$ which is geodesic. An embedded wheel graph $H$ with center $x$ is called 
geodesic, if there are two adjacent points $a,b$ on the boundary of $H$, 
such that for any $c$ on $H$, the geodesic triangle $abc$ has maximal length among all $c \in S(x)$. 
These {\bf geodesic wheel sub-graphs} do not need to be unique for a given triangle $x,a,b$.
In the graph but play the role of the geodesic disk 
$H=\exp_x(D)$ in the continuum. A compact manifold $M$ or a finite simple graph $G$ can now be declared to have 
{\bf positive $\mu$-curvature}, if all sectional curvatures are positive for the measure $\mu$. 

\paragraph{}
In dimension $2$, all wheel sub-graphs are also geodesic and any 
triangulation of a positive curvature 2-manifold has positive curvature in the above sense.
Imposing a definite $\mu$-curvature sign on {\bf all wheel graphs} would be 
a very strong condition in dimensions 3 and higher. 
It would exclude negative curvature of dimension 3 or higher because intersections of spheres 
are 2-spheres and contain wheel graph which by Gauss-Bonnet have positive curvature. We therefore
need to restrict the class of wheel graphs which play the role of geodesic discs $H=\exp_x(D)$.
In the continuum, the notion is natural because if $\mu$ is the natural volume measure 
on the linear functions on an ambient space of a Nash embedding of a Riemannian
manifold $M$, it leads to the classical sectional curvature for $M$. 

\paragraph{}
The Hopf conjectures \cite{Hopf1932,Hopf1946,Hopf1953,BishopGoldberg,BergerPanorama}
can now be studied also for $2d$-graphs. The first adaptation of the classical Hopf conjecture to the discrete is:
{\bf does positive curvature for some $\mu$ on $\Omega(G)$ 
imply positive Euler characteristic $\chi(G)$?} Also: {\bf does negative curvature for a $2d$-graph 
lead to an Euler characteristic with the sign of $(-1)^d$?}
If we would ask positive curvature for the uniform measure of all colorings on all wheel graphs,
this would imply that wheel graphs have $5$ or $6$ vertices.
The graph $G$ then necessarily is a $2d$-sphere \cite{SimpleSphereTheorem}. 
By restricting the wheel graphs to geodesic wheel graphs, the positive curvature condition 
becomes less strong.

\paragraph{}
The second Hopf question {\bf whether there a positive curvature metric on $S^2 \times S^2$} 
can be ported to the discrete too. If we implement $S^2 \times S^2$ as a finite simple $4$-graph, 
we can ask whether there exists a measure $\mu$ on locally injective functions such that
we have a positive $\mu$ curvature for all geodesic wheel graphs. 
This is a {\bf linear programming problem}. It can be asked for example for a 
concrete graph with $26^2$ vertices, which is the product of the octahedron graph $O$ with 
itself \cite{KnillKuenneth} and where the Barycentric refinement of $O$ has 26 vertices.

\paragraph{}
\cite{BergerPanorama} displays the Hopf conjecture as one of 4 pillar problems of Hopf guiding
him in his early research.  1. {\bf the Hopf sign conjecture} 
states that the sign of the Euler characteristic $\chi(M)$ of a $2d$ manifold is $(-1)^d$ in 
the negative curvature case and positive in the positive curvature case, 
2. {\bf The product enigma} asking whether the product manifold $S^{2d} \times S^{2d}$ admits 
a positive curvature metric for positive $d$. In the book \cite{GromollKlingenbergMeyer}
the question is raised, whether in general for any two compact manifolds $M,N$ the product
$M \times N$ never allows for an everywhere positive sectional curvature. There are some remarks
on \cite{Bourguignon75} and that under some symmetry, there is no positive curvature metric
\cite{AmannKennard}. 
3. {\bf The Gauss-Bonnet problem} aimed to generalize Gauss-Bonnet to 2d-manifolds. (It was solved by Hopf
himself in the hypersurface case, then solved first by 
Allendoerfer, Fenchel, then by Allendoerfer, Weil for manifolds embedded in an ambient space and finally by 
Chern intrinsically \cite{HopfCurvaturaIntegra,Allendoerfer,Fenchel,AllendoerferWeil,Chern44,Chern1990}).
(Allendoerfer and Fenchel would with Nash get the full result too and \cite{AllendoerferWeil} made use 
already of a local embedding.  4. The {\bf pinching problem} asked which pinching conditions
do force $M$ to be a sphere? This was solved by \cite{Rauch51}, 
Berger-Klingenberg \cite{BergerPanorama,GromollKlingenbergMeyer}
and finally \cite{BrendleSchoen} in the smooth category.

\paragraph{}
It is a bit vexing how little of the original work is usually cited when mentioning the Hopf 
conjecture. Usually, the book \cite{BergerPanorama} is referred to which cites papers of Hopf
from the 1920ies. \cite{BishopGoldberg} which states the conjecture on the first page does not
mention Hopf at all. We tried a bit harder.
To our knowledge, the Hopf conjecture appears first stated explicitly in a talk given by Hopf in 
1931 \cite{Hopf1932}. Hopf has worked on topics related to that already in the 20ies and might have 
been guided by the questions earlier on but it appears not visible in the writings of the 20ies, 
to our knowledge even so the work on generalizing Gauss-Bonnet might have been motivated by it.
The question is also mentioned in a Festschrift to Brouwers 60'st birthday and 
was written early in 1941 \cite{Hopf1946}. 
A statement in a 1953 talk given in Italy mentions the relation with sign. 
Interestingly, the explicit sign conjecture is stated as a question in \cite{BishopGoldberg} 
without referring to any work of Hopf. Also Chern has proven earlier some result towards it, 
\cite{Chern1966} (page 169),
suggesting that he not only has been aware of the sign conjecture but also worked on it.

\paragraph{}
A strong link between the continuum and discrete has always been integral geometry.
Crofton formulas allow to recover measures like distances using integrals.
For literature on integral geometry, see \cite{Santalo,Santalo1,KlainRota,Schneider1}, in geometry
\cite{Banchoff67,Banchoff1970}. Integral geometry has been used as a glue between the 
discrete and continuum before an intermediate step being piecewise flat spaces 
\cite{CheegerMuellerSchrader} which was of interest also to mathematical physics due to 
Regge calculus \cite{Regge}. 
For Morse theory, see \cite{Mil63,NicoalescuMorse}.
We have used some integral geometry related to curvature in graph theory:
\cite{indexexpectation,indexformula, KnillBaltimore,colorcurvature,AmazingWorld}. 

\section{Sectional curvature} 

\paragraph{}
The definition of sectional curvature in Riemannian manifold $M$ uses the exponential map
$\exp$, a concept that is based on the existence of a geodesic flow. Important is to restrict the sprays 
to two-dimensional planes in the tangent plane $T_xM$ as the curvature tensor in the continuum can be obtained
from these sectional curvatures. The question appears now, 
how to port this to the continuum, where we have no linear structure and already for points $y$ in distance $2$
away from a given point $x$, more than one geodesic exists connecting $x$ with $y$. The intersection
$S(x) \cap S(y)$ is in general a $(d-1)$ simplex which for $d \geq 2$ has more than $1$ point.
One can look at all embedded wheel graphs centered at $x$ representing the Grassmannian $Gr(2,T_xM)$,
but that would be much too strong. 

\paragraph{}
We have seen in \cite{SimpleSphereTheorem} that a naive positive curvature assumption is quite restrictive.
If we ask for positive curvature that every wheel graph has 4 or 5 boundary points and so to
curvature $1-4/6 = 1/3$ and $1-5/6=1/6$, positive 
curvature graphs using that notion leads directly to spheres, the reason being (and this is similar to the
Berger-Klingenberg sphere theorem), because the graph is then the union of two balls. 
It can surprise a bit that no projective planes are allowed. The reason why no projective
plane is possible is because there is not enough space to implement this. In some sense, the strong curvature condition 
$K=1/3$ or $K=1/6$ produces already a pinching condition. For negative curvature graphs, the 
notion of sectional curvature can not exist for the simple reason that for any 
$d$-graph for $d \geq 3$ there are always positive curvature wheel graphs embedded. By Gauss-Bonnet, they are
always present in unit spheres of points of d-graphs. 

\paragraph{}
In order to define positive sectional curvature more realistically in the discrete, we have to 
select a class of {\bf geodesic wheel graphs}. They need to play the role of $H = \exp_x(D)$ with 
$D \subset Gr(2,T_xM)$. As just mentioned, there is a need to restrict the set of wheel graphs 
which qualify to play the role of the geodesic sheets $H$. Figure~(\ref{wheel}) illustrates how
wild a wheel graph could look like, even if it is embedded nicely. We obviously have to restrict the
class of embedded wheel graphs which qualify as replacements for elements in $Gr(2,T_xM)$. 

\begin{figure}[!htpb]
\scalebox{0.6}{\includegraphics{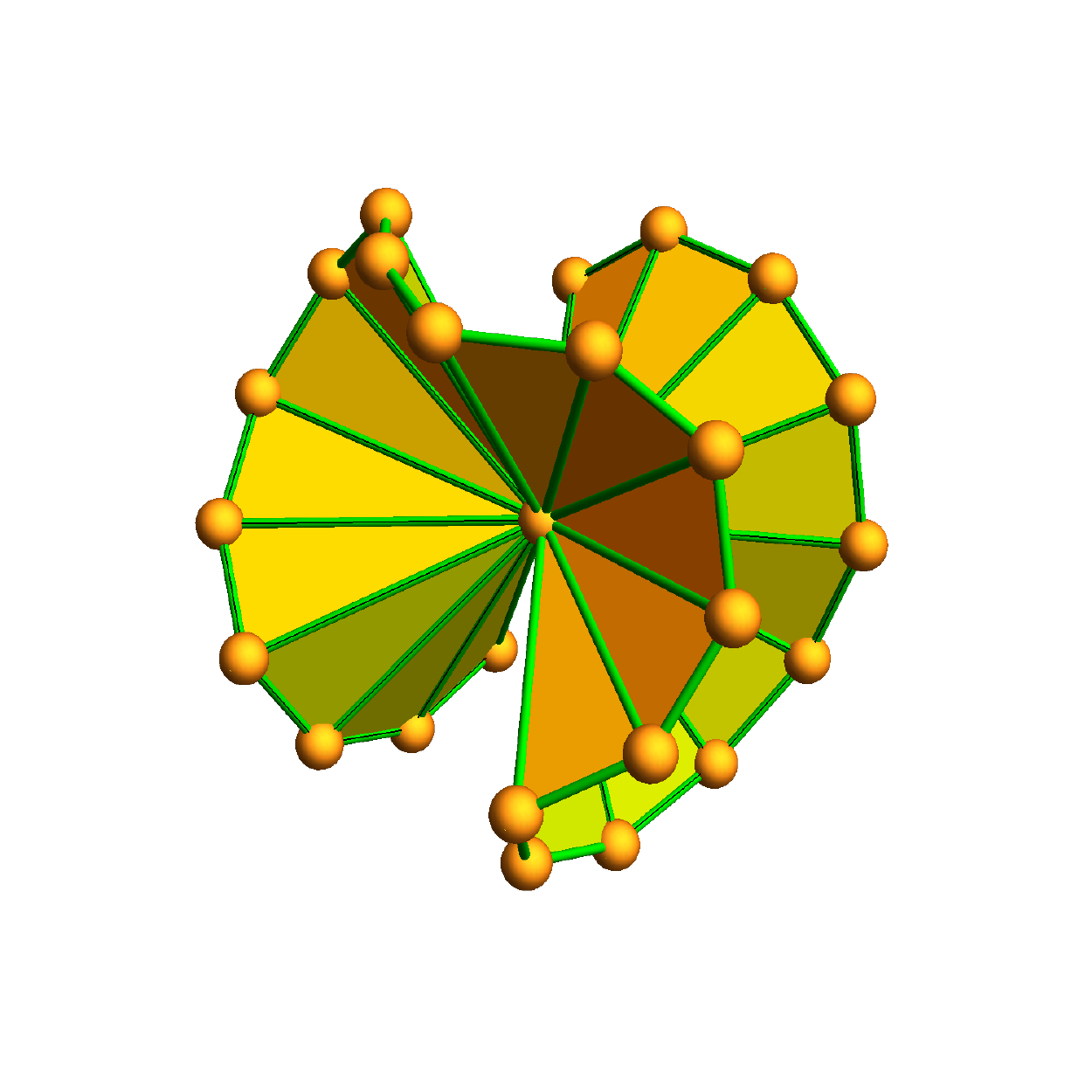}}
\label{wheel}
\caption{
Embedded $2$-dimensional wheel graphs in $d$-graphs can be wild
and can in general not serve as the analogue of geodesic surfaces $H = \exp_x(D)$.
Requiring all wheel graphs to have positive curvature produces only very 
small graphs. This was shown in \cite{SimpleSphereTheorem}. In dimension $2$, there are
only 6 positive curvature surfaces with this strong definition. In general, they are
all spheres. 
}
\end{figure}

\paragraph{}
Lets look at a small geodesic sphere $S_r(x)$ in a Riemannian manifold  $M$
and intersect a small sphere $S_r(x)$ with $H = \exp_x(D)$, where $D$ is a disc in a $2$-dimensional 
plane in $T_xM$. It produces a $1$-dimensional circle $C$ in the $(d-1)$-dimensional sphere $S_r(x)$. This 
circle is close of being a closed geodesic in $S_r(x)$. It is not actually a closed geodesic in general
because  $S_r(x)$ is a $(d-1)$-dimensional Riemannian manifold diffeomorphic to a sphere which in general 
admits only a finite number of closed periodic orbits, so that for most planes $H$ given by an
element of the Grassmanian ${\rm Gr}(2,T_xM)$, the circle $S_r(x) \cap H$ is not a closed periodic
geodesic circle in $S_r(x)$. Still, as it is close, this is what we can ask for in the discrete.

\paragraph{}
A circular subgraph $C$ of a unit sphere $S(x)$ in a $d$-sphere $S$ is called a {\bf geodesic grand circle}, 
if for for any triangle  $x,a,b$ with $a,b \in C$, and every other point $c$ in $S(x)$, all three connections 
$a \to b, b \to c, c \to a$ in $C$ are minimal geodesics
and every $c \in C$ has the property that it maximizes $d(a,c) + d(b,c)$ among all points $c \in S(x)$. 
We do not claim that for a given $a,b$ the ``grand circle" $C$ is unique, it is not in general but the 
length of $C$ is determined. The usual curvature $1-|C|/6$ is determined but 
the $\mu$-curvature is not.

\paragraph{}
In the continuum, the definition means that $C$ is a geodesic triangle through $a,b$. It might not actually be
a closed geodesic as most pairs of points are not part of a geodesic circle if the sphere is not round. 
But for any two close points $a,b$ in a sphere $S(x)$ we can look at a geodesic connecting $a,b$ 
on a different side than the short connection. In the continuum, we can define it as the geodesic 
through $a,b$ which minimizes the total angle difference between end velocity vectors. 
Also in the discrete, we might have difficulty to prove that $c$ exists if we do not assume $a,b$ to be 
neighboring points. Take any geodesic triangle
connecting $a,b,c$ is what we call a {\bf geodesic circle}. 
Together with $x$ the circle produces a geodesic wheel. It plays the role of $H={\rm exp}_x(D)$ in 
the continuum. 

\section{The Hopf sign conjecture} 

\paragraph{} 
The classical Hopf conjecture asking whether an even-dimensional positive curvature
Riemannian manifold $M$ has positive Euler characteristic $\chi(M)$ or whether a
negative curvature $2d$-Riemannian manifold $M$ has positive $(-1)^d \chi(M)$
is maybe the most popular open question in 
global Riemannian geometry and perhaps even in the entire field of Riemannian geometry
\cite{Berger2002}. 

\paragraph{}
Despite the elegance and urgency of the Hopf conjectures, they have remained open. 
The original Hopf conjecture follows in $2$-dimensions by Gauss-Bonnet, in 
$4$ dimensions by the {\bf Synge theorem}, Hurewicz and Euler-Poincar\'e 
and in the case of hypersurfaces $M$ because $M$ is then a sphere by convexity of the enclosed 
region in the ambient space $E$. Under a pinching condition, the conjecture follows from 
{\bf sphere theorems} like the Rauch-Berger-Klingenberg sphere theorem. In dimension $4$
also the algebraic argument with the Gauss-Bonnet-Chern integrand $K$ works. As Milnor noticed,
some clever coordinates allow to write 
$K=K_{12} K_{34} + K_{13} K_{42} + K_{14} K_{23} + R_{1234}^2+R_{1342}^2 +R_{1423}^2$ in terms
of sectional curvatures and squares of Riemann curvature tensor entries. 
Already in dimension 6 and higher the stronger algebraic 
question fails as an example of Geroch has shown \cite{Geroch,Klembeck}.  

\paragraph{}
A bit stronger than the Hopf conjecture is the following conjecture. The 
assumption is now not that the sectional curvatures are positive but that there
is a measure $\mu$ for which all geodesic sheets $H = \exp_x(D)$ have 
positive curvature at $x$ with respect to the index expectation induced 
on $H$. We say then that $M$ has {\bf positive sectional $\mu$ curvatures.}
 
\conjecture{
If $M$ is a $2d$ Riemannian manifold $M$ with positive sectional $\mu$ curvatures, 
then it has positive Euler characteristic. }

\paragraph{}
This is stronger because it implies the classical Hopf conjecture:
a general compact Riemannian manifold can be Nash embedded into an 
ambient Euclidean space $E=\mathbb{R}^n$. 
The natural $SO(n)$-invariant measure on the unit sphere $S$ of $E$ 
induces now a measure $\mu$ on linear functions $f(x)=v \cdot x$ with $v \in S$. 
The-$\mu$ curvature of a geodesic disc $\exp_x(D) \subset M$ with a 
$2$-dimensional disc in $T_xM$ is then equal to the induced $\mu$ 
curvature. So, positive curvature implies positive $\mu$-curvature for this
particular $\mu$. The above conjecture would then imply 
$\chi(M)>0$ and settle the original Hopf conjecture. It is certainly also
possible, that the original conjecture is true and the stronger conjecture fails. 

\paragraph{}
One reason for making the conjecture stronger is that the more general version allows
continuous deformations of the measure $\mu$, allowing to use deformation techniques. 
Deformations of the Riemannian metric have been tried early \cite{Bourguignon75} and
also Ricci flow techniques suggest new possibilities as the differentiable sphere theorem 
has shown. In the discrete however, where we are in a combinatorial setting and have 
no Riemannian metric to deform, we have a measure $\mu$ which can be deformed 
in a simplex of measures. When restricting to finite dimensional simplicies,
we can use linear programming algorithms. 

\paragraph{}
The conjecture can now be explored for graphs. The analogue conjecture is: 

\conjecture{ 
If $M$ is an even-dimensional $d$-graph $G$ of positive sectional $\mu$-curvatures, 
then it has positive Euler characteristic.}

\paragraph{}
Allowing for more general curvatures allows also to revive and explore a 
{\bf more general algebraic Hopf conjecture} for Riemannian manifolds: 

\conjecture{
If $M$ is an even-dimensional positive curvature Riemannian manifold $M$ 
(in the classical sense), there exists a measure $\mu$ on the space of Morse
functions $\Omega$ for which the Euler curvature $K_{\mu}(x)$ is positive. }

\paragraph{}
If answered positively, this would imply that $\chi(M)=\int_M K_{\mu}(x) dx>0$ and settle the Hopf conjecture.
The original {\bf algebraic Hopf conjecture} \cite{Weinstein71} was confirmed in dimension $4$ by Milnor
(\cite{Chern1955} Theorem 5 credits an oral communication of Milnor) 
but it turned out to fail in dimension $6$ or higher by counter examples given by 
Geroch \cite{Geroch,Klembeck} so that positive curvature does not imply positive Gauss-Bonnet-Chern curvature 
in general. Allowing for more flexibility with respect to the definition of sectional and 
Euler curvature allows to explore more general questions which can be ported to the discrete: 

\conjecture{
If $M$ is an even-dimensional graph $G$ of positive $\mu$-curvature, then there exists an 
other measure $\nu$ for which the Euler curvature $K_{\nu}(x)$ is positive.
}

\paragraph{}
This again would imply $\chi(M)=\sum_M K_{\mu}(x) > 0$ and settle the discrete Hopf conjecture.
Why is this exciting? Because this can be checked in principle on every given graph using 
a finite procedure. It is a linear programming problem: start with a measure $\mu$ giving
positive curvature, now look whether one can modify $\mu$ to achieve positive Euler curvature
at every point. It might well be that the modification will destroy positive sectional curvature
at some places, but who cares? We are eventually only interested in the result, the Euler
characteristic of $G$. The deformation could be done similarly than a Ricci flow in a
space of measures. This space of measures is a finite dimensional polyhedron in a 
linear space of measures. 

\section{The Hopf product conjecture}

\paragraph{}
Finally, let us look at a discrete analogue of the second Hopf conjecture, which is a
statement about $S^2 \times \S^2$. Classically, the conjecture is that
$S^2 \times S^2$ can not carry a positive curvature metric. The analogue stronger question is: 

\conjecture{
The manifold $S^2 \times S^2$ does not admit a $\mu$-measure of positive curvature. 
}

\paragraph{}
As mentioned in \cite{Berger2002}, there is a more general {\bf question of Yau} from 1982 who
observed that no simply connected non-negative curvature manifold is known for which it is proven
that there is no positive curvature metric. Here is the adaptation of the Yau question to 
graphs: 

\conjecture{
Is there a simply connected $d$-graph $G$ admitting a $\mu$ measure of non-negative curvature 
for which one can show that it has no $\mu$ measure of positive curvature? }

\paragraph{}
The $S^2 \times S^2$ Hopf conjecture in the discrete becomes:

\conjecture{
No graph implementation of $S^2 \times S^2$ admits a $\mu$-measure of positive curvature. 
}

\paragraph{}
Given any concrete graph, this is a linear programming problem. 
We can get a $\mu$-curvature on $S^2 \times S^2$ which has constant 
Euler curvature. We need to know whether this can be deformed so that
all sectional curvatures are positive. 

\section{Positive curvature 2-graphs}

\paragraph{}
A $2$-graph is a finite simple graph for which every unit sphere is a $1$-sphere, 
that is a circular graph with $4$ or more vertices. It is realized as a convex 
polyhedron of positive curvature if the angle sum $\alpha(x)$ at every vertex 
satisfies $K(x)=1-\alpha/(2\pi)>0$. 
A polyhedron can be seen as a topological 2-manifold in which the curvature is a pure point (Dirac) measure 
supported on finitely many points and such that every face is a triangle. We
say then that the graph $G$ is {\bf realized as a positive curvature polyhedron} 
in $\mathbb{R}^3$. The Gauss-Bonnet theorem tells $\sum_x K(x) = \chi(G)$. 
Here is the discrete Hopf statement in two dimensions:
if a finite simple $2$-graph $G$ with $n$ vertices is realized as a positive curvature polyhedron, 
then it has positive Euler characteristic. As the convex hull in the geometric realization is a ball, 
we know even that $G$ is a 2-sphere. The reverse is true. If $G$ is a 2-sphere it can be realized 
as a positive curvature polyhedron. 

\paragraph{}
The notion of positive curvature allows also to put positive curvature on the discrete 
projective plane. There is an implementation of the projective plane with 15 vertices for which
we constructed by hand an average of 65 index functions $i_f(x)$ which add up to a 
positive curvature. The strategy was to make a list of 15 index functions with 3 critical
points each such that $f_j(j)=1$, then take combinations to get all coordinates positive. 
This linear programming problem could be solved by hand. In general, we might want to 
use a computer and the simplex algorithm. 

\paragraph{}
Are there other cases beside the two-dimensional cases, where the discrete 
Hopf conjecture is verified?  We can verify it for hypersurfaces in Euclidean spaces. 

\section{About the classical Hopf conjecture}

\paragraph{}
A smooth compact even dimensional Riemannian manifold $M$ is defined to have {\bf positive curvature}
if all sectional curvatures are positive. Heinz Hopf conjectured that such a 
manifold $M$ must have positive Euler characteristic $\chi(M)$. This is the famous {\bf Hopf conjecture}.
In two dimensions, the statement is well known, as the classification of surfaces
shows that any positive curvature manifold is either a sphere or a projective plane, which both 
have positive Euler characteristic. In dimension $2$ it follows also by Gauss Bonnnet as sectional 
curvature agrees then with Euler curvature. The 2-sphere has Euler characteristic $2$
the projective plane, the quotient of a $\mathbb{Z}_2$ action on the 2-sphere
has Euler characteristic $1$. One can assume the manifold to be orientable
in Hopf as taking a double cover doubles the Euler characteristic. 

\paragraph{}
The definition of sectional curvature for a two dimensional plane $P$ in $T_xM$ spanned by $v,w$ is 
$(K(x,v,w)=R(v,w,v,w)/|v \times w|^2$, where $R$ is the Riemann curvature tensor at $x$. 
A tensor-free definition is the Bertrand-Puiseux formula $\lim_{r \to 0} (|C_0(r)|-|C(r)|) 3\pi/r^3$,
where $C_0(r)$ is a circle of radius $r$ in the flat plane and $C(r)$ is a circle of radius $r$ in
the 2-dimensional surface ${\rm exp}_x(P)$. [ One of the difficulties when defining the analogue of
geodesic patches $\exp_x(D)$ is to describe what circular subgraphs $S(x)$ should be the analogue
of $\exp_x(D) \cap S_r(x)$.  ]

\paragraph{}
Sphere theorems give an other hint: if the manifold is orientable and the sectional
curvatures are sufficiently pinched, meaning that the maximal
and minimal curvature ratio is smaller than 4 at every point, then the manifold
is a sphere and as even dimensional spheres have all Euler characteristic $2$ by the Schl\"afli \cite{Schlafli}
formula $\chi(S^n) = 1+(-1)^n$ the Hopf conjecture holds under a pinching condition. The
differentiable sphere theorems even give more hints: the Ricci flow smooths out a positive curvature
manifold and deforms it to a sphere. However, the Ricci flow techniques are difficult; it is necessary 
to regularize the flow for example. It is not inconceivable that some sort of curvature heat flow smooths
the geometry and deforms a positive curvature manifold to some generalized space form, where Euler curvature
is positive everywhere. None has been found so far, but there is a nagging suspicion that there could be a 
matching deformation flow. But such a deformation argument might not exist. Worse yet, the conjecture 
could be wrong. 

\paragraph{}
There is a substantial difference between non-negative curvature and positive curvature but the
product Hopf conjecture explains that the differences are not yet understood. 
Many examples are given in \cite{Ziller}. 
A non-negative curvature manifold can be the flat $2$-torus for example or then 
$S^2 \times S^2$, which as a product of two positive curvature manifolds has non-negative curvature. 
The Hopf question about positive curvature metric on $S^2 \times S^2$.
can be asked for graphs also: can we find an 
average of Poincar\'e-Hopf indices on a discrete $S^2 \times S^2$ (for the Cartesian product in graphs,
see \cite{KnillKuenneth} such that the product of a $d$ graph and a $q$ graph a $p+q$ graph), 
such that this curvature is positive? 
If that is the case, then very likely also in the continuum that some index average on the manifold is positive. 

\section{Attacks on the Hopf conjecture}

\paragraph{}
The Hopf product problem like putting a positive curvature metric on $S^2 \times S^2$ is less clear.
Results like \cite{Weinstein70} telling that no positive metric
implementation can be embedded in $\mathbb{R}^6$ suggest that no positive metric can exist after all.
The sign conjecture is reasonable for many reasons. In \cite{Weinstein70} it is also shown that
if the $2d$ manifold is embedded in $R^{2d+2}$, then $b_2=0$ and $\chi(M)>0$. 
It holds in dimension $2$ and $4$ and in the positive curvature case under a pinching condition by the
sphere theorems. Also, if the $2d$-manifold $M$ is a product of smaller dimensional 2-manifolds with 
positive curvature then because the curvatures multiply $K(x_1, \cdots, k_d) = \prod K(x_i)$ one has
a positive Euler integrand and positive Euler characteristic, also because $\chi(M) = \prod_i \chi(M_i)$
by general principles. The paper \cite{Bourguignon75} mentions a curious relation between the product 
conjecture and the sign conjecture: if a positive curvature metric would exist on $S^3 \times S^3$, then
this would give a negative answer to the sign conjecture.

\paragraph{}
An early approach the Hopf conjecture has been through Gauss-Bonnet-Chern but that failed.
This is not the end of the approach however as one can imagine that some other curvature satisfying Gauss-Bonnet
is positive under a positive curvature assumption. Any probability measure $\mu$ on the space $\Omega$ of
Morse functions produces a {\bf curvature} as {\bf index expectation} $K(x)={\rm E}[i_f(x)]$, 
where $i_f(x)$ is the Poincar\'e-Hopf index satisfying $\sum_x i_f(x) = \chi(M)$. 
It is not excluded that some average of indices produces a curvature
which is positive for positive curvature manifolds. 
(Weinstein excluded curvatures coming from curvature tensors satisfying some 
semi-algebraic conditions.) But the integral geometric curvatures obtained by averaging Poincar\'e-Hopf indices
are not necessarily of this form. While this would settle the Hopf conjecture, no
such measure has emerged yet. Integral geometric methods in differential geometry trace back the mathematical 
genealogy lines like Banchoff-Chern-Blaschke. \cite{Banchoff67} 

\paragraph{}
An other approach is to find a Morse function which has only critical points of positive indices $i_f(x) \geq 0$,
where $i_f(x)=1$ at all critical points. 
Poincar\'e-Hopf $\chi(G) = \sum_x i_f(x)$ gives from this then positive Euler characteristic and
would settle the conjecture. The idea for such a function $f$ comes from the idea of embedding
the surface in a surrounding space which has a natural function $f$ with nearly flat level curves. 
This approach borrows from the fact that if $M$ is a positive curvature hypersurface in an odd dimensional 
space, then indeed, there is such a function and the Hopf conjecture is true. 

\paragraph{}
The Morse inequalities $b_k(G) \leq c_k(G)$, where $b_k$ are the Betti numbers and $c_k$ the minimal number
of critical points a Morse function can have of Morse index $k$, assures even that the odd Betti 
numbers are zero (implying again
by Euler-Poincar\'e $\chi(G) = \sum_k (-1)^k b_k >0$). If one would find a Morse 
function $f$ with just two critical points, the maximum and minimum, then by Reeb's theorem or by Lusternik-Schnirelmann,
one would have a sphere. Since there are positive curvature manifolds which are not spheres, this statement is 
definitely too strong. But one can try to exclude critical points with odd index. 
We have so far not seen any even dimensional orientable positive curvature manifold which does not 
allow a Morse function with only even Morse indices and so only even non-zero Betti numbers. 

\paragraph{}
Cohomology is an other attack angle. 
A bit weaker is the statement that positive curvature implies that all odd Betti numbers are zero. 
The four dimensional complex projective plane has the Betti vector $\vec{b}=(1,0,1,0,1)$ so that the Morse 
inequalities assure that any Morse function must have a critical point $x$ of Morse index $m(x)=2$. 
This complex two plane $\mathbb{C} P^2)$ which is homeomorphic to $S^5/S^1$, is an example,
where the geometry forces any Morse function to have a critical point of 
Morse index $2$. Equipped with the natural metric, this manifold satisfies the
quarter pinching condition showing that the sphere theorems
are sharp with respect to the pinching constant. 

\paragraph{}
The four dimensional case is special also as the Hopf conjecture is then already a consequence of the 
Synge theorem \cite{Synge}. The Synge theorem assures that an orientable manifold of positive curvature is simply 
connected and if non-orientable has fundamental group $\mathbb{Z}_2$. 
This is a result from calculus of variations: a homotopically non-trivial closed geodesic is a critical point
of a variational problem. The analogue of the second derivative test which is given in the form of the
Jacobi equations shows however that this can not be a minimum. A ``rubber-band" on an orientable positive curvature
surface slips off as it can be made smaller by moving it. A rubber band on a doughnut however is stable if it is placed at the
most narrow part of the doughnut. The non-orientability also prevents a some rubber bands to slip off a projective plane.
In four dimensions, also Gauss-Bonnet worked as the algebraic Hopf conjecture holds there. 

\paragraph{}
The Morse case can be written a bit differently. On an even dimensional manifold, we can 
write the symmetric index $j_f(x)=(i_f(x)+i_{-f}(x))/2$ as
$j_f(x) = 1-\chi(B_f(x))/2$, where $B_f(x) = S(x) \cap \{ f(y) = f(x)\}$ is the center manifold at $x$. 
In the Morse case, at a critical point, $B_f(x)$ is then either empty or a product of two spheres.
The negative indices can occur only if both factors have even dimension. In the simplest case of
$d=2$, the center manifold $B_f(x)$ at a critical point is either empty or the product $S_0 \times S_0$
of two zero dimensional spheres. Indeed, at a maximum or minimum, $B_f(x)$ is empty and at a saddle point
the level curve through the origin is a crossed pair of lines which intersects a small geodesic circle
$S_r(x)$ in 4 points. 

\section{An illustration using hypersurfaces}

\paragraph{}
To illustrate the idea of preventing classes of critical points which leads via Morse inequalities to 
bounds on the Betti numbers, 
let us prove the following elementary {\bf hyper surface sphere theorem}.
already investigated early by Hopf \cite{HopfCurvaturaIntegra}.
We restrict to $4$-manifolds. In general, for positive curvature 4 manifolds, the relation $b_1=0$ by 
the simply connectedness (Synge) and $b_3=0$ by Poincar\'e duality so that the Hopf conjecture in 4 dimensions 
follows already from the fact that all odd Betti numbers zero zero leading to a homology sphere and so Euler
characteristic $2$ as only $b_1=b_4=1$ and the others are zero. But we want to illustrate the center manifold method. 
A hypersurface $M$ of positive curvature bounds a convex region in $\mathbb{R}^{d+1}$ so 
that it bounds a solid sphere and must be a sphere, but the we want to explain the local method. 

\paragraph{}
Here is a weak Synge type result establishing $b_1=0$ for hypersurfaces:  \\
There can not be any critical point with Morse index $1$ or $3$ if $M$ is a compact $4$-manifold
with positive curvature which is embedded as a hypersurface in $E=\mathbb{R}^5$.

\begin{proof}
If there is such a point $x$, the center manifold $B_f(x)$ is disconnected as it is $S^{k-1} \times S^{4-k-1}$ for $k=1$
or $k=3$ which means $B_f(x)=S_2 \times S_0$, the union of two $2$-spheres. They are asymptotically symmetric with respect
to $x$. Without loss of generality, we can assume $f(y)>f(x)$ for $y$ in the interior of the spheres
$B_f(x) = \{ y \; | \; f(y)=f(x), d(x,y)=r\}$. Take a two dimensional plane through $x$ which is tangent at points
$y,y'$ to the two spheres, where the two spheres are on different sides.
Positive curvature implies that $f(\phi y)>0$ and
$f(\phi y')>0$. But one of the points $\phi(y), \phi(y')$ is in the interior of $B_f(x)$, the other outside.
This is a contradiction. Therefore, it is not possible that $B_f(x)$ is disconnected.
Having excluded disconnected $B_f$, we have excluded critical points of Morse index $1$ or $n-1$.
The Morse inequalities tell now that $b_1=-0$. By Poincar\'e-duality, one has $b_3=0$. 
\end{proof}

{\bf Hopf for hypersurfaces}: 
If $M$ is an $4$-dimensional compact positive curvature 
Riemannian manifold which is realized as a smooth hypersurface 
in the Euclidean space $E=\mathbb{R}^5$, then any ambient linear function $f$ in $E$ 
which is Morse on $M$ has only critical points with positive index. Either by Poincar\'e-Hopf or the
Morse inequalities, positive Euler characteristic follows. 

\begin{proof}
One knows that for almost all linear functions $f$ 
in the ambient space $E$, the induced function on $M$ is a Morse function.
Lets just take a linear function $f$ in $E$ which induces
a Morse function $f$ on $M$. We show that $f$ can not have indices of Morse index $1$ or $3$. 
The symmetric index $j_f(x)=(i_f(x)+i_{-f}(x))/2$
is  $1-\chi(B_f(x))/2$, where $B_f(x) = S(x) \cap \{ f(y) = f(x)\}$ is the center manifold.
This center manifold is a $2$-dimensional sphere in the regular case (leading to $j_f(x)=1-2/2=0$) 
or a product of two spheres $B_f = S_2 \times S_0$ or $B_f= S_0 \times S_2$ (for Morse index $1$ or $3$) 
or then empty for maxima or minima (Morse index $0$ or $4$) or then $B_f = S_1 \times S_1$ in the case of 
Morse index $2$. \\ 
To prove the result in the 4 dimensional case we only have to exclude $S_2 \times S_0$ and $S_0 \times S_2$ 
as critical manifolds $B_f(x)$. \\
Take a two dimensional plane $P$ in through $x$ in $E$ which is tangent to $B_f(x)$. 
On this plane, $f$ is constant as it is defined on a two dimensional plane tangent to a surface 
$B_f(x)$ on which $f$ is constant and because $f$ is linear. The space $B_f(x)$ is disconnected.
The two dimensional plane will in general not be tangent to the other component. 
But the plane is $r^3$ close on be tangent on the other connected component of $B_f(x)$. 
Now apply the exponential map to $P$ at $x$, (this is defined as $P$ can be naturally be identified
with part of the tangent space $T_xM$). This produces a non-linear surface $H = \exp_x(D)$ which is in $M$ and 
on $S_r(x)$ of order $r^2$ away from $P$ on one side. But $H \cap S_r(x)$ is empty
as its function values changes everywhere to a either positive or negative value except at $x$. 
\end{proof} 

\paragraph{}
In two dimensions, $B_f(x)$ is empty for the maxima and minima and the 4-point space 
$S_0 \times S_0$ at a saddle point. The later does not happen for positive curvature. 
It has also been realized by Hopf already that any positive curvature manifold which is not a sphere can not be written as a hypersurface
in an Euclidean space. Also, as Hopf remarked, it implies things like that the projective plane can not be embedded
into $\mathbb{R}^3$ or that 
the complex plane $\mathbb{C}P^2$ can not be embedded into $\mathbb{R}^5$. 
This $4$-manifold is homeomorphic to $S^5/S^1$ and Betti number $(1,0,1,0,1)$.

\paragraph{}
The notion of center manifold $B_f(x)$ can also be considered in the discrete. 
For an even dimensional d-graphs, we can write the symmetric index $j_f(x)$  as
$j_f(x) = 1-\chi(B_f(x))/2$, where $B_f(x) = \{ S(x) \cap \{y \; | \;  f(y)=f(x) \}$
is the central manifold at $x$. 
In the Morse case, $B_f(x)$ is either a sphere of dimension $d=2$ (the regular case)
or the product of two spheres and so has $\chi(B_f(x)) \in \{ 0,4 \}$. 
If $f$ induces on $B_f(x)$ a Morse function, then we can compute $\chi(B_f(x))$ by
computing more central manifolds. Now, we can say when $G$ has positive curvature.
In the case $d=2$, the graph has positive curvature if there is a probability measure $\mu$
on Morse functions for which $K(x)=E[j_f(x)]$ is positive everywhere. Positive curvature
graphs in the old sense have positive curvature but there can be more. 

\section{Constructing measures}

\paragraph{}
Here is a construction of a curvature $K(x)$ on $M$ which is positive on most of $M$. 
The construction works locally in a geodesic patch, where we can use geodesic coordinate systems.
The problem of doing that globally over the entire manifold is not done
as it would prove the Hopf conjecture. Here is a proposition which is not so surprising seeing
that one can find a diffeomorphism $\phi$ from $M$ to an other manifold $N$ which is mostly a sphere,
take from there the measure $\mu$ obtained from a Nash embedding which leads to the Gauss-Bonnet-Chern
curvature which is then positive, the use $\phi$ to produce Morse functions on the ambient space of $M$. 
It is not so much the result which is interesting, it is the proof which links sectional curvatures with
$\mu$-curvature. 

\begin{propo}
Given a Riemannian manifold $M$ of positive curvature and an open patch $U=\exp_x(D)$ with a ball $D=B_r(0) \subset T_xM$ 
of radius smaller than the injectivity radius of $M$ at $x$, then there is a measure $\mu$ on Morse functions of $M$
such that the index average curvature $K=K_{\mu}$ is positive everywhere in $U$. 
\end{propo}

\paragraph{}
\begin{proof}
If $f$ is in the set $\Omega$ of Morse functions on $M$, 
the Poincar\'e-Hopf theorem $\sum_x i_f(x) = \chi(M)$	
expresses the Euler characteristic $\chi(M)$ as a sum of its Poincar\'e-Hopf indices  $i_f(x)$.
Any probability space $(\Omega,\mu)$ of Morse functions then defines a curvature by expectation $K(x) = E[i_f(x)]$. 
Fubini leads to Gauss-Bonnet $\int_M K(x) \; dV(x) = \chi(M)$, where the volume measure $dV$ is normalized 
to be a probability measure on $M$.  \\

By Nash, $M$ can be isometrically embedded in an ambient Euclidean space $E$. The set of linear functions 
$f(x)=v \cdot x$ with $v$ in the unit sphere $S$ of $E$ induces a probability space $(\Omega,\mu)$ if $\mu$
comes from the normalized volume probability measure on $S$. 
Almost every $f \in \Omega$ is Morse restricted to $M$ and induces a Morse function on any geodesic disc $H=\exp_x(D)$,
where $D \subset T_xM$ is a small disk in a $2$-plane of $T_xM$. The index expectation
curvature at $x \in H$ has the same sign than the Gauss curvature of the surface $H$ and 
so the sectional curvature of $H$. 
By assumption, these sectional curvatures are all either positive or negative. \\

Write $z=x$ with $z_j=(x_{2j-1},x_{2j})$.
A point $(w=(w_1, \cdots, w_d),f_1,\dots,f_d)$ in the probability space $M \times \Omega^d$ defines
a new function $F$ in $\Omega$. Its definition is local. 
[A parametrized partition of unity could make it global,
using Sard to keep it Morse but there is a problem with controlling the 
curvature in the intersection of patches.]
Near a given $x \in M$, find geodesic coordinates 
$z=(z_1,\dots,z_d)$ in a neighborhood $U$. Define for any $w \in M$ and $z \in U$ a new Morse function
$F(z_1, \dots, z_d) = f_1(z_1, w_2,\dots,w_d) + \cdots + f_d(w_1,\dots,w_{d-1},z_d)$ in $U$.
Having a map $M \times \Omega^d \to \Omega$ defines by push forward of $dV \times \mu^d$ 
a new measure $\nu$ on $\Omega$. Expectation of $i_F(x)$ over $(\Omega,\nu)$ defines a new curvature $K(x)$.
A critical point $z=(z_1, \dots, z_d)$ for $F$ necessarily is simultaneously a critical point for the functions 
$z_k \to g_k(z_k)=f_k(w_1,\dots,w_{k-1},z_k,w_{k+1},\dots,w_d)$ and the index $i_F(z)$ is the product 
$i_F(z) = \prod_k i_{g_k}(z_k)$. Unlike the Gauss-Bonnet-Chern integrand which by Geroch can fail to have 
a definite sign for $d \geq 3$ and some positive curvature $M$, the curvature $K$ has the right sign.   \\

Under the positive curvature assumption, we have 
$E[i_{g_k}(z_k)] > 0$ for all $k$ and all $z,w$. 
For negative curvature, we have $E[i_{g_k}(z_k)]<0$ for all $k$.
Because the random variables $i_{g_k}(z_k)$ are independent on $M \times \Omega^d$
they are also uncorrelated, so that $K(z) = E[i_F(z)] = \prod_{k=1}^d E[i_{g_k}(z_k)]$
which is positive for positive curvature and negative for odd $d$ and negative curvature.
Gauss-Bonnet now assures that $\int_M K(z) \; dV(z)>0$ in the positive curvature case and 
that $(-1)^d \int_M K(z) \; dV(z)>0$ in the negative curvature case. 
\end{proof}

\begin{figure}[!htpb]
\scalebox{0.6}{\includegraphics{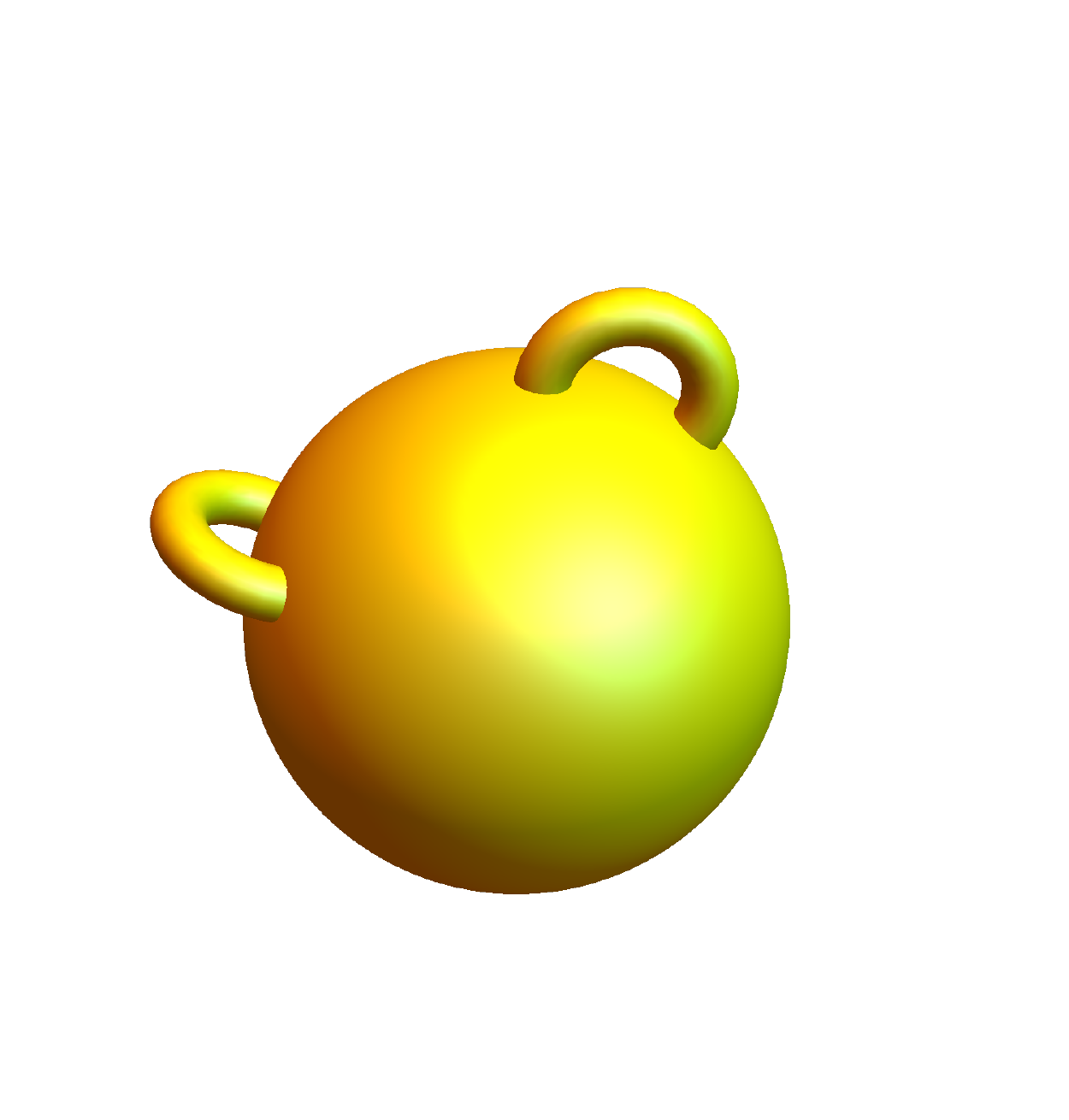}}
\label{sphere}
\caption{
By embedding a suitably deformed manifold into an ambient
Euclidean space one can establish that any manifold can 
have positive Gauss-Bonnet-Chern integrand on an arbitrary 
large part of the manifold. That statement is not so interesting.
What is interesting is a construction of $K_{\mu}$ which involves
the sectional curvatures. 
}
\end{figure}

\paragraph{}
We get so the following statement:

\begin{coro}
Given a positive curvature $2d$-manifold $M$ and $\epsilon>0$. There exists a subset $U$
of measure larger than ${\rm Vol}(M)-\epsilon$ and a $\mu$ curvature $K$ satisfying the Gauss
Bonnet-Chern relation $\int_M K \; dV = \chi(G)$ such that $K$ is positive on $U$. 
\end{coro}

\begin{proof}
By doing the construction of the proof independently in each $U_k$ of
a finite disjoint collection $U_k=\exp_{x_k}(B_{r(x)})$ 
with $r(x_k)$ is smaller than the injectivity radius at $x_k$, there exists a
$\mu$-curvature $K$ which is positive in $U$. Since the $U_k$ are disjoint, we can 
form a partition of unity of $M$ where only one of the functions $\rho_k$  $1$ on each of 
the sets $U_k$ and zero on the other sets. If $F_k$ are the Morse functions constructed
on $U_k$ as done in the proposition, then by parametrizing the partition of unity we
can get (using Sard, see in the more detailed form given by \cite{NicoalescuMorse}) 
that most of the gluing produces indeed global Morse functions even-so 
they must become wild in $M \setminus \bigcup_k U_k$. 
\end{proof}

\paragraph{}
This does by no means solve the Hopf conjecture as the curvature $K$ could be very 
negative in the remaining part of measure $\epsilon$ and render
$\int_M K \; dV$ negative. And we have currently no idea how to control the 
curvature there. Note for example that we can for any 2-surface produce a metric
that has positive curvature on most of the manifold: take a sphere and attach tiny 
handles to it. The handles can have arbitrary small area and still produce negative
Euler characteristic. 

\paragraph{}
We are not aware how to deform the Riemannian 
metric to achieve this for the Gauss-Bonnet-Chern integrand and Geroch's example shows that 
such an attempt is futile in general. 
Constructing more general measures allowed us to do that. 
Unfortunately,  we can not control the places where the different patches are glued with 
the partition of unity. The different patches would have to be aligned nicely in order to 
control the curvature there. It is not excluded however that one can make the argument 
global for certain classes of manifolds. 

\section{Discrete and Continuum}

\paragraph{}
Finally, we want to look at the problem whether the discrete and continuous results
can be linked in some way. Given a compact Riemannian $2d$-manifold $M$ of positive curvature.
Let $G=(V,E)$ be a finite simple $2d$-graph which triangulates $M$. Assume that $M$ is isometrically
embedded in an Euclidean space. We can assume that the vertices $v \in V$ are points in $M$ and so
in $E$. Let $\mu$ be the usual rotational homogeneous rotationally invariant measure on linear functions.
Almost all linear functions $f(x)=v \cdot x$ in $E$ with $|v|=1$ are both Morse functions 
on $M$ as well as locally injective function on $G$. Now, if $G$ is fine enough the sectional $\mu$ 
on $G$ as defined here is positive too. If the sign Hopf conjecture is true for graphs, then $\chi(G)>0$
but as $\chi(G)=\chi(M)$ if the triangulation is fine enough, we also have $\chi(M)>0$. 

\paragraph{}
This argument is convincing but it needs to be made more precise. There is previous work which does
similar things like \cite{CheegerMuellerSchrader}. Let us state this

\conjecture{
If the Hopf conjecture holds for finite simple $2d$-graphs $G$ with positive $\mu$ sectional curvature, then 
the Hopf conjecture holds for Riemannian $2d$-manifolds $M$ with positive $\mu$ sectional curvature. 
}

\paragraph{}
One can also look at the other direction. Assume the Hopf conjecture is true for manifold,
does it hold for graphs? This is less clear as one can imagine small graphs where 
things are different. The question might anyway be mute as both cases could be true. 

\paragraph{}
We expect that also the sphere theorems work in the same way as in the continuum. 

\conjecture{
Assume a simply connected Riemannian manifold $M$ allows
for a $\mu$ curvature that has positive sectional curvatures and
which satisfies the quarter pinching condition, then 
$M$ is diffeomorphic to a $d$-sphere. 	
}

In the discrete, one can ask

\conjecture{
Assume a finite simple graph $G$ has a $\mu$ curvature which leads to positive sectional
curvatures and which satisfies the quarter pinching condition, then $G$ is a $d$-sphere
(as a d-graph). 
}

\paragraph{}
This is expected to be difficult to prove, unlike the simple case, where
we ask every embedded wheel graph to have positive curvature (meaning that there are
4 or 5 boundary points in the wheel graph). In that case, the simple connectivity is 
not necessary and the graphs are very small \cite{SimpleSphereTheorem}.

\paragraph{}
The question relating the discrete and the continuum also is relevant when looking at 
metric properties. Given a measure $\mu$ one can by the Cauchy-Crofton formula also 
define a distance between two points. This can also be used in the discrete.
A minimizing path $\gamma$ connecting two points $a,b$ in a graph does in general not exist
uniquely. But given a locally injective function $f \in \Omega$ on the graph we can look at 
the random variable $N_{\gamma}(f)$ counting the number of edges in $\gamma$ where $f$
changes sign. This defines the pseudo metric
$$ d(x,y) = \inf_{\gamma(x,y)} E[N_{\gamma}]  \; , $$
where the infimum is taken over all curves connecting $x$ with $y$ with the understanding that
$d(x,y) = \infty$, if $x,y$ are not connected. 
The {\bf Kolmogorov quotient} of all equivalence classes of the equivalence relation 
given by $d(x,y)=0$ defines a metric space. For reasonable probability spaces $(\Omega,\mu)$
the Kolmogorov quotient is not needed. An example is if the graph $G$ is embedded in some
large dimensional $\mathbb{R}^n$. Taking the space of linear functions, one gets the
length of the polygonal path. By perturbing the measure $\mu$ one can achieve however
that geodesic connections are in general unique.

\bibliographystyle{plain}

\end{document}